\documentclass[twoside]{article}
\usepackage{graphicx,amssymb,mathrsfs,amsmath,enumerate}
\usepackage{amsthm}
\usepackage{charter}
\usepackage[dvipdfm]{hyperref}
\catcode`@=11
\long\def\@makefntext#1{\noindent #1}
\newskip\tabcentering \tabcentering=1000pt plus 1000pt minus 1000pt
\def\REF#1{\par\hangindent\parindent\indent\llap{#1\enspace}\ignorespaces} 
\def\MCH#1#2{\setbox0=\hbox{\raise#1\hbox{#2}}\smash{\box0}}

\def\@evenfoot{}\def\@oddfoot{}

\def\@evenhead{\hbox to\textwidth{\footnotesize\rm\thepage \hfill
{\it Cheng Zhiyun}}} 

\def\@oddhead{\hbox to \textwidth{\footnotesize{\it
When is region crossing change an unknotting operation?} \hfill\thepage}}

\catcode`@=12

\def\bc{\begin{center}}
\def\ec{\end{center}}
\def\no{\noindent}
\def\hang{\hangindent\parindent}
\def\textindent#1{\indent\llap{\qquad #1\ \ \enspace}\ignorespaces}
\def\ref{\par\hang\textindent}

\newtheorem{theorem}{Theorem}[section]

\newtheorem{corollary}[theorem]{Corollary}

\newtheorem{proposition}[theorem]{Proposition}

\begin{document}
\abovedisplayskip=6pt plus 1pt minus 1pt \belowdisplayskip=6pt
plus 1pt minus 1pt
\thispagestyle{empty} \vspace*{-1.0truecm} \noindent
\vskip 10mm

\bc{\large\bf When is region crossing change an unknotting operation?\\[2mm]
\footnotetext{\footnotesize The authors are supported by NSF 11171025 and Scientific Research Foundation of Beijng Normal University}} \ec

\vskip 5mm
\bc{\bf Cheng Zhiyun\\
{\small School of Mathematical Sciences, Beijing Normal University
\\Laboratory of Mathematics and Complex Systems, Ministry of
Education, Beijing 100875, China
\\(email: czy@mail.bnu.edu.cn)}}\ec

\vskip 1 mm

\noindent{\small {\small\bf Abstract} In this paper, we prove that region crossing change on a link diagram is an unknotting operation if and only if the link is proper. A description of the behavior of region crossing change on link diagrams is given. Furthermore we also discuss the relation between region crossing change and the Arf invariant of proper links.
\ \

\vspace{1mm}\baselineskip 12pt

\no{\small\bf Keywords} region crossing change; unknotting operation \ \

\no{\small\bf MR(2000) Subject Classification} 57M25\ \ {\rm }}

\vskip 1 mm


\vspace{1mm}\baselineskip 12pt



\section{Introduction}
In this paper. we consider some local transformations on link diagrams. In [4], H. Murakami defined $\sharp$-operation and showed that $\sharp$-operation is an unknotting operation. In [5], $\triangle$-unknotting operation was defined by H. Murakami and Y. Nakanishi. At a later time, Y. Nakanishi [6] proved that a $\triangle$-unknotting operation can be obtained from a finite sequence of 3-gon moves. Hence 3-gon move is also an unknotting operation. In [1], Haruko Aida generalized 3-gon moves to $n$-gon moves, which was also proved to be an unknotting operation, see the figure below.
\begin{center}
\includegraphics{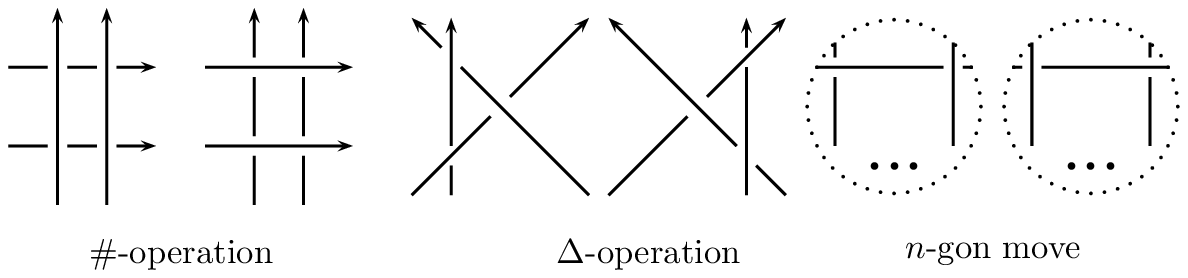} \centerline{\small Figure
1\quad}
\end{center}

Recently, a new local transformation on link diagram was introduced in [9], named as region crossing change. Here a \textit{region crossing change} at a region of $R^2$ divided by a link diagram is defined to be the crossing changes at all the crossing points on the boundary of the region. For example, the figure below shows the effect of taking region crossing change on the region with capital letter R:
\begin{center}
\includegraphics{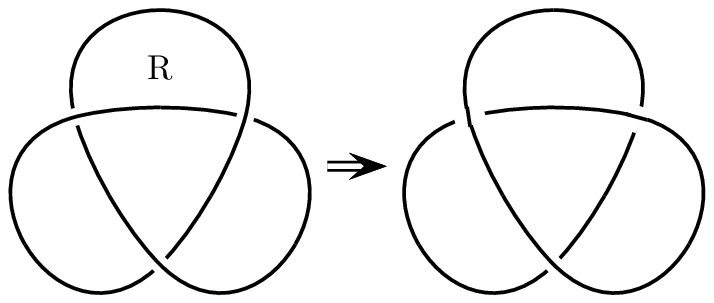} \centerline{\small Figure
2\quad}
\end{center}

Evidently, $\sharp$-operation and $n$-gon move mentioned above are both special cases of region crossing changes. Therefore we say region crossing change is an \textit{unknotting operation} on a link diagram if there exist some regions of $R^2$ divided by the link diagram such that if we apply region crossing changes on these regions the new diagram represents a trivial link. We remark that during the process Reidemeister moves are forbidden, i.e. the diagram are kept if we regard it as a 4-valent graph and ignore the information of the crossings. For the case of knots, the theorem below was proved in [9].
\begin{theorem}$^{[9]}$
Let $D$ be a knot diagram and $p$ a crossing point of $D$, then there exist some regions such that if one takes region crossing changes on these regions, $D$ will be transformed into a new knot diagram $D'$, here $D'$ is obtained from $D$ by a crossing change at $p$.
\end{theorem}

Obviously it follows that region crossing change is an unknotting operation on knot diagrams. In general, region crossing change is not always an unknotting operation for link diagrams. For instance, the standard diagram of Hopf link can not be transformed into a diagram of trivial link by region crossing changes, since the two crossing points are both on the boundary of each region of the diagram. Hence a natural question is on which kind of link diagrams, region crossing change is an unknotting operation. In [2], we give an answer to this question for 2-component links.
\begin{theorem}$^{[2]}$
Region crossing change is an unknotting operation on a diagram of $L=K_1\cup K_2$ if and only if $lk(K_1, K_2)$ is even.
\end{theorem}

In this paper, we will prove the following theorem, which can be regarded as a generalization of the theorem above.
\begin{theorem}
Region crossing change is an unknotting operation on a diagram of $L=K_1\cup K_2\cup \cdots\cup K_n$ if and only if
\begin{center}
$\sum\limits_{j\neq i} lk(K_i, K_j)=0$ $(mod$ $2)$
\end{center}
for all $1\leq i\leq n$.
\end{theorem}

We say a link is \textit{proper} if it satisfies the condition in Theorem 1.3. In [5], H. Murakami and Y. Nakanishi proved that $L=K_1\cup K_2\cup \cdots\cup K_n$ can be obtained from $L'=K'_1\cup K'_2\cup \cdots\cup K'_n$ by a finite sequence of $\sharp$-operations (some Reidemeister moves may be needed) if and only if
\begin{center}
$\sum\limits_{j\neq i} lk(K_i, K_j)=\sum\limits_{j\neq i} lk(K'_i, K'_j)$ $(mod$ $2)$
\end{center}
for all $1\leq i\leq n$. Similarly, in [1], the author proved that $L=K_1\cup K_2\cup \cdots\cup K_n$ can be deformed into $L'=K'_1\cup K'_2\cup \cdots\cup K'_n$ by a finite sequence of $n$-gon moves (some Reidemeister moves may be needed) if and only if \begin{center}
$\sum\limits_{j\neq i} lk(K_i, K_j)=\sum\limits_{j\neq i} lk(K'_i, K'_j)$ $(mod$ $2)$
\end{center}
for all $1\leq i\leq n$. Since when we talk about the equivalence generated by region crossing changes, Reidemeister moves are forbidden, hence Theorem 1.3 can not be obtained from the two results above evidently. For the same reason, we can only discuss the necessary condition of the equivalence up to region crossing changes, see Proposition 4.1. The sufficient condition does not make sense unless we are given two link diagrams which are isotopic as 4-valent graphs.

The rest of the paper are arranged as follows: in Section 2 we will take a brief review of the incidence matrix defined in [2] and some related results of it. In Section 3 we will prove Theorem 1.3 for 3-component links, which is the initial step of the induction used in Section 4. In Section 4, we will give the proof of Theorem 1.3 and offer a complete solution to detect whether some given crossing points of a link diagram can be switched by region crossing changes. Hence the behavior of region crossing change on link diagrams are well understood. Finally the relation between region crossing change and the Arf invariant is discussed.

\section{Incidence matrix of a link diagram}
In this section we will take a quick review of incidence matrix which was defined in [2]. Given a link diagram $D$, let $G$ and $G'$ be the Tait graph of $D$ and the dual graph respectively. In graph theory [10], the incidence matrix of a graph is defined as below
\begin{center}
$M(G)=(m_x(y)),\quad x\in V(G)$ and $y\in E(G)$
\end{center}
and
\begin{center}
$m_x(y)=
\begin{cases}
1& \text{if $y$ is incident with $x$;}\\
0& \text{otherwise.}
\end{cases}$
\end{center}
If we use $M(G)$ $(M(G'))$ to denote the incidence matrix of $G$ $(G')$, since $G$ and $G'$ have the same size, we can obtain a new $(c+2)\times c$ matrix from $M(G)$ and $M(G')$ as below
\begin{center}
$M(D)=\begin{bmatrix}
 M(G)   \\
 M(G')  \\
\end{bmatrix},$
\end{center}
here $c$ denotes the number of crossing points of $D$. If we work with $\mathbb{Z}_2$ coefficients, it is not difficult to find that the incidence matrix $M(D)$ is closely related to region crossing changes. In fact each row of $M(D)$ corresponds to a region of $D$, and the positions of 1's of one row tell us which crossings will be changed if we take region crossing change on the corresponding region. Moreover, given a set of regions, in order to understand the effect of region crossing changes on these regions, one just need to read the positions of 1's on the sum of the corresponding rows. The following theorem was proved in [2], here the rank means the $\mathbb{Z}_2$-rank.
\begin{theorem}$^{[2]}$
Let $L$ denote a $n$-component link, and $D$ a diagram of $L$, then the rank of $M(D)$ equals to $c-n+1$, here $c$ denotes the crossing number of $D$.
\end{theorem}

Before ending this section, we want to fix two conventions we will use throughout and mention two useful propositions about region crossing changes. First all diagrams mentioned in this paper are non-split. Besides, sometimes we will abuse our notation, letting $L=K_1\cup K_2\cup \cdots\cup K_n$ refer both to a link diagram and the link itself, so is each component $K_i$ of $L$. It is not difficult to determine the precise meaning from context. Given a diagram $D$ of $L=K_1\cup K_2\cup \cdots\cup K_n$, we define a set of crossing points of $D$, say $P$, are \textit{region crossing change admissible} if we can obtain a new link diagram $D'$ from $D$ by a sequence of region crossing changes, here $D'$ is obtained from $D$ by taking crossing changes on every crossing point of $P$. Then we have
\begin{proposition}
Let $L$ be a link diagram, and $L_1$ is a sub-link of $L$. Choose a set of crossing points of $L_1$, say $P$, if $P$ is region crossing change admissible on the diagram of $L_1$, then it is also region crossing change admissible on the diagram of $L$.
\end{proposition}

\begin{proof}
Notice that any region of $L_1$ is the union of some regions of $L$, if there is no nugatory crossing, then the effect of region crossing changes on the union of these regions is equivalent to the effect on the original one. If there exist some nugatory crossings, with the skill of handling nugatory crossings, see Proposition 2.1 in [2], we can still suitably choose some regions of $L$ which satisfy our requirement. This finishs the proof.
\end{proof}

Finally we want to recall a result in [2].
\begin{proposition}$^{[2]}$
Given an $n$-component link diagram $L=K_1\cup \cdots \cup K_n$, each crossing point of $K_i\cap K_i$ $(1\leq i\leq n)$ is region crossing change admissible, and each pair of crossing points of $K_i\cap K_j$ $(1\leq i< j\leq n)$ are region crossing change admissible.
\end{proposition}

\section{The case of 3-component links}
In this section we will prove Theorem 1.3 for the case of 3-component links.
\begin{proposition}
Region crossing change is an unknotting operation on a diagram of $L=K_1\cup K_2\cup K_3$ if and only if $L$ is a proper link.
\end{proposition}

\begin{proof}
Let us consider the sufficient part first. Now $L$ is a proper link and assume $D$ is a diagram of $L$, with crossing number $c$. Let $P$ denote an unknotting set of crossing points, i.e. if one takes crossing changes on all points of $P$ then the new diagram represents a trivial link. Obviously if $lk(K_i, K_j)$ is odd (even), then $(K_i\cap K_j)\cap P$ contains odd (even) crossing points. Since $L$ is proper, we can divide our discussion in three cases:
\begin{itemize}
\item $K_1\cap K_3=\varnothing$.

Since $D$ is non-split and $L$ is proper, it follows that $lk(K_1, K_2)=lk(K_2, K_3)=0$ $($mod 2$)$. Then according to Proposition 2.3, we conclude that $P$ is region crossing change admissible.
\item $K_1\cap K_2\neq\varnothing$, $K_2\cap K_3\neq\varnothing$, $K_3\cap
    K_1\neq\varnothing$ and $lk(K_1, K_2)=lk(K_2, K_3)=lk(K_3, K_1)=0$ $($mod 2$)$.

    As above, $P$ is also region crossing change admissible in this case.
\item $K_1\cap K_2\neq\varnothing$, $K_2\cap K_3\neq\varnothing$, $K_3\cap K_1\neq\varnothing$ and $lk(K_1, K_2)=lk(K_2, K_3)=lk(K_3, K_1)=1$ $($mod 2$)$.

In this case we claim that $\forall p_1\in K_1\cap K_2$, $\forall p_2\in K_2\cap K_3$ and $\forall p_3\in K_3\cap K_1$, $\{p_1, p_2, p_3\}$ are region crossing change admissible. Then combining Proposition 2.3, the conclusion follows. According to the orientation of $K_1$ and $K_2$, we smooth $p_1$ such that $K_1$ and $K_2$ become one component, say $K'$. Now there are only two components in the new diagram, $K'$ and $K_3$. Therefore it follows from Proposition 2.3, $\{p_2, p_3\}$ are region crossing change admissible on $K'\cup K_3$. In other words, there exist some regions of $D$ such that taking region crossing changes on them, $\{p_2, p_3\}$ will be changed. If these region crossing changes also changes $p_1$, then these regions satisfy our requirement. Otherwise we smooth $p_2$ and $p_3$ respectively. If all these three cases can not offer some regions as required, then together with Proposition 2.3 it means that for any pair points of $(K_1\cap K_2)\cup(K_2\cap K_3)\cup(K_3\cap K_1)$, they are region crossing changes admissible. Hence with those rows of $M(D)$, we can construct a matrix as below
\begin{center}
$\begin{bmatrix}
 1& &&&&&& \\
 & \ddots &&&&&&\\
 & & 1& &&&& \\
 &&&1 & 1 &   &  &   \\
 &&&  & 1 & 1 &  &   \\
 &&&  &   &\ddots&\ddots&  \\
 &&&  &   &  & 1 & 1 \\
 &&&1  &   &   &   & 1\\
\end{bmatrix}$,
\end{center}
where the top left identity submatrix corresponds to those self-crossing points, i.e. the crossing points of $K_i\cap K_i$, and the right bottom submatrix is referred to those crossing points between different components. It is obvious that the rank of this matrix is $c-1$, hence the rank of $M(D)$ is at least $c-1$, which contradicts with Theorem 2.1. Hence we finish the proof of the sufficient part.
\end{itemize}

Now we turn to the proof of the necessary part. Assume $L$ is not a proper link, there are two possibilities:
\begin{itemize}
\item $K_1\cap K_3=\varnothing$. It follows that $lk(K_1, K_2)$ and $lk(K_2, K_3)$ can not be both even. Without loss of generality, we suppose $lk(K_1, K_2)$ is odd.

    If $lk(K_2, K_3)$ is odd. Since region crossing change is an unknotting operation, we conclude that for any $p_1\in K_1\cap K_2$ and $p_2\in K_2\cap K_3$, $\{p_1, p_2\}$ are region crossing change admissible. It means that any pair of non-self-crossing points are region crossing change admissible. Then we can construct a matrix with those rows of $M(D)$ as above. The contradiction follows.

    If $lk(K_2, K_3)$ is even. Since region crossing change is an unknotting operation, therefore any crossing point of $K_1\cap K_2$ is region crossing change admissible. One can also construct a matrix as above, which also leads to a contradiction.
\item $K_1\cap K_2\neq\varnothing$, $K_2\cap K_3\neq\varnothing$, $K_3\cap K_1\neq\varnothing$. Without loss of generality, we assume $lk(K_1, K_2)$ is odd, and $lk(K_2, K_3)$ is even. We continue our discussion in two cases.

    If $lk(K_3, K_1)$ is odd. Since for any $p_1\in K_1\cap K_2$, $p_2\in K_2\cap K_3$ and $p_3\in K_3\cap K_1$, $\{p_1, p_2, p_3\}$ are region crossing change admissible. If region crossing change is an unknotting operation then any crossing point of $K_2\cap K_3$ is region crossing change admissible. It follows that any pair crossing points of $(K_1\cap K_2)\cup(K_3\cap K_1)$ are region crossing change admissible. Similarly we can obtain a contradiction as above.

    If $lk(K_3, K_1)$ is even. In this case any crossing point of $K_1\cap K_2$ is region crossing change admissible. The contradiction follows similarly.
\end{itemize}

In conclusion, if $L$ is not proper then region crossing change is impossible to be an unknotting operation. The proof is finished.
\end{proof}

It is worth noting that the proof of the necessary part is direct, all the possible cases of a 3-component link are discussed. In Section 4, Proposition 4.1 will offer a solution of it with another viewpoint.

\section{The proof of the main theorem}
Before giving the proof of Theorem 1.3, we need some preliminary results.
\begin{proposition}
If a link diagram $L=K_1\cup K_2\cup \cdots\cup K_n$ can be obtained from another link diagram $L'=K'_1\cup K'_2\cup \cdots\cup K'_n$ by a sequence of region crossing changes, then $($after suitably ordered if necessary$)$
\begin{center}
$\sum\limits_{j\neq i} lk(K_i, K_j)=\sum\limits_{j\neq i} lk(K'_i, K'_j)$ $(mod$ $2)$
\end{center}
for all $1\leq i\leq n$.
\end{proposition}

\begin{proof}
It suffices to show that for any $1\leq i\leq n$, $\sum\limits_{j\neq i} lk(K_i, K_j)$ $($mod 2$)$ is unaffected by one region crossing change. In fact, it is easy to observe that given a region of the diagram, there are even crossing points on the boundary that are generated by $L-K_i$ and $K_i$. Consequently $\sum\limits_{j\neq i} lk(K_i, K_j)$ $($mod 2$)$ is invariant, then the result follows.
\end{proof}

The next proposition plays an important role in the proof of the main theorem, it can be regarded as a generalization of Proposition 2.3.

\begin{proposition}
Given a link diagram $L$, $\{K_1, \cdots, K_n\}$ are some components of it. If $K_i\cap K_{j}\neq\varnothing$ for all $\{i, j\}$ which satisfy $|i-j|=1$ or $|i-j|=n-1$, then for any crossing point $p_1\in K_1\cap K_2, \cdots, p_{n-1}\in K_{n-1}\cap K_n, p_n\in K_n\cap K_1$, $\{p_1, p_2, \cdots, p_n\}$ are region crossing change admissible.
\end{proposition}

\begin{proof}
When $n=1$ or 2, the statement follows from Proposition 2.3, the case $n=3$ follows from Proposition 2.2 and the claim in Proposition 3.1. Now we assume the statement is correct for $n\leq k$, it suffices to show it is also correct for $n=k+1$.

If there exist a pair of integers $\{i, j\}$ with $1<j-i<k$, such that $K_i\cap K_j\neq \varnothing$, then we can choose a crossing point $q$ from $K_i\cap K_j$. According to the assumption, $\{p_1, \cdots, p_{i-1}, q, p_{j}, \cdots, p_{k+1}\}$ and $\{p_i, \cdots, p_{j-1}, q\}$ are both region crossing change admissible. As a result, $\{p_1, p_2, \cdots, p_{k+1}\}$ are region crossing change admissible.

If for any $\{i, j\}$ which satisfy $1<j-i<k$, $K_i$ and $K_j$ have no intersection, let us consider the diagram of the sub-link $L'=K_1\cup \cdots \cup K_{k+1}$. According to Proposition 2.2, it is sufficient to prove $\{p_1, p_2, \cdots, p_{k+1}\}$ are region crossing change admissible on the diagram of $L'$. Similar to the proof of Proposition 3.1, we can smooth $p_{k+1}$ to obtain a $k$-component link diagram $L''$. By induction, $\{p_1, p_2, \cdots, p_k\}$ are region crossing change admissible on the diagram of $L''$. If the corresponding region crossing changes of $L'$ will affect $p_{k+1}$, then the result follows. Consequently we only need to consider the case whichever crossing point of $\{p_1, p_2, \cdots, p_{k+1}\}$ is chosen to smooth, the corresponding region crossing changes will not affect itself. In other words, any $k$ points of $\{p_1, p_2, \cdots, p_{k+1}\}$ are region crossing change admissible on the diagram of $L'$. Due to these facts, we can use the rows of $M(L')$ to construct a $c\times c$ matrix $M$ as below, here $c$ denotes the crossing number of $L'$.
\begin{center}
$M=\begin{bmatrix}
 I& &&&& \\
 & A &&&&\\
 & & \ddots& && \\
 &&&A &  &      \\
 &B  & \cdots & B & I &   \\
 &B&\cdots& B &   & I\\
\end{bmatrix}$,

where $I=\begin{bmatrix}
 1& && &\\
 & 1 &&&\\
 & & \ddots& & \\
 &&&1& \\
 &&&&1\\
\end{bmatrix}$, $A=\begin{bmatrix}
1 & 1 &   &  &   \\
  & 1 & 1 &  &   \\
  &   &\ddots&\ddots&  \\
  &   &  & 1 & 1 \\
1 &   &   &   & 1\\
\end{bmatrix}$ and $B=\begin{bmatrix}
1 &  &   &  &   \\
1  &  &  &  &   \\
\vdots &   &&&  \\
1  &   &  &  & \\
1 &   &   &   & \\
\end{bmatrix}$.
\end{center}
As a result, $rank M(L')\geq rank M=c-k+1$, which contradicts the fact $rank M(L')=c-(k+1)+1=c-k$. The proof is finished.
\end{proof}

Now we are going to turn to the proof of Theorem 1.3.
\begin{proof}
The necessary part directly follows from Proposition 4.1, therefore it suffices to prove the sufficient part of the theorem. If $lk(K_i, K_j)$ is even for all $1\leq i<j\leq n$, the result follows from Proposition 2.3.

Therefore we assume there exist some pairs of components with odd linking number. Let us construct a graph $G$ which contains $n$ vertices, each vertex $v_i$ corresponds to a component $K_i$ of $L$. If $lk(K_i, K_j)$ is odd, then we add an edge between $v_i$ and $v_j$. Let $P$ be a unknotting set of $L$. By Proposition 2.3, we can change all the self-intersections and each pair of crossing points between two components in $P$. Thus for any pair of components with odd linking number, there exists one crossing point between them in the remainder of $P$. Hence what we want to do is to remove all the edges from $G$ by region crossing changes. Notice that Proposition 4.2 tells us that a loop of $G$ can be removed by region crossing changes, therefore we can remove loops one by one. Because $L$ is proper, then for any vertex $v_i$ there are even vertices which are adjacent to $v_i$. Hence the process will continue until all the edges have been removed. The proof is finished.
\end{proof}

As we mentioned in Section 1, we can talk about two link diagrams being related by finite region crossing changes only if these two link diagrams are isotopic as 4-valent graphs. Under this condition, it can be proved analogously that $L=K_1\cup K_2\cup \cdots\cup K_n$ and $L'=K'_1\cup K'_2\cup \cdots\cup K'_n$ are related by a sequence of region crossing changes if and only if
\begin{center}
$\sum\limits_{j\neq i} lk(K_i, K_j)=\sum\limits_{j\neq i} lk(K'_i, K'_j)$ $(mod$ $2)$
\end{center}
for all $1\leq i\leq n$.

Given a link diagram $L=K_1\cup K_2\cup \cdots\cup K_n$ and some crossing points of it, say $Q$, a natural question is whether $Q$ is region crossing change admissible? In order to answer this question, we just need to switch all the crossing points of $Q$, then $Q$ is region crossing change admissible if and only if $L$ and the new link $L'$ satisfy the condition above. Or we can construct a graph with vertices $\{v_1, \cdots, v_n\}$ which correspond to the components $\{K_1, \cdots, K_n\}$ respectively. For each non-self-intersection point of $Q$, we add an edge between the corresponding two vertices. Finally we obtain a graph $G(L; Q)$, then it is evident that $Q$ is region crossing change admissible if and only if each vertex of $G(L; Q)$ has even valency.

\section{Region crossing change and Arf invariant}
The aim of this section is to study the relation between region crossing change and Arf invariant. According to Theorem 1.3, region crossing change is an unknotting operation on $L$ if and only if $L$ is proper. Proper links are very special since Arf invariant is well defined on them. Hence a natural question arises: is there any relations between region crossing change and Arf invariant? Before discussing this question, we take a short review of the proper link and its Arf invariant.

Recall that we say a link $L$ is a \emph{proper link} if for any component of $L$, the sum of the linking numbers between this component and the rests is an even integer. According to [7], we can define the Arf invariant of a proper link in this way: let $M=S^3\times [0, 1]$, then $\partial M=S^3\times \{0\}\cup S^3\times \{1\}=\partial M_+\cup \partial M_-$. Given a proper link $L$ and a knot $K$ which are embedded in $\partial M_+$ and $\partial M_-$ respectively, if there exists a regularly embedded 2-manifold $N$ of genus zero such that $\partial N\cap \partial M_+=L$ and $\partial N\cap \partial M_-=K$, then we say $K$ is \emph{related} to $L$. It was proved in [7] that if $K$ and $K'$ are two knots related to the same proper link $L$, then Arf$(K)=$Arf$(K')$. Therefore we can define Arf$(L)\triangleq $Arf$(K)$ where $K$ is a knot related to $L$.

In practice, given a proper link $L=K_1\cup K_2\cup \cdots\cup K_s$ $($without loss of generality, we assume that $L$ is non-split$)$, in order to calculate the Arf invariant of $L$, we can handle it as follows. First choose a crossing point between $K_i$ and $K_j$, then smooth it according to the orientations of $K_i$ and $K_j$$($see the figure below$)$. Now we obtain a proper link with $s-1$ components. Repeating this process until we get one component, i.e. a knot $K$. By the definition above, we have Arf$(L)=$Arf$(K)$.
\begin{center}
\includegraphics{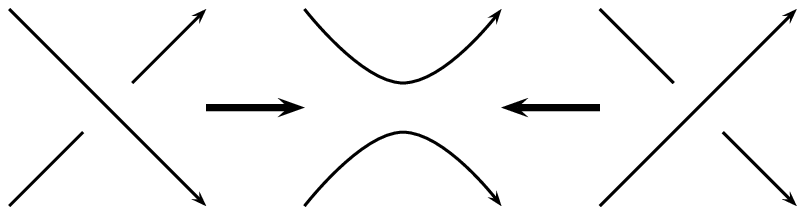} \centerline{\small Figure
3\quad}
\end{center}

In [4], it was shown that one $\sharp$-operation changes the Arf invariant of the knot, i.e. if $K$ and $K'$ are related by one $\sharp$-operation, then Arf$(K)+$Arf$(K')=1$. In [5], a similar result was given for $\triangle$-operation, i.e. if $K$ and $K'$ are related by one $\triangle$-operation, we also have Arf$(K)$+Arf$(K')=1$.

Let $L$ be a diagram of a proper link, and $R$ a region of it. After taking region crossing change on $R$, one obtain a new proper link $($Proposition 4.1$)$, say $L'$. Now we want to investigate the relation between Arf$(L)$ and Arf$(L')$.

Consider the region $R$, we denote the crossing points on the boundary of $R$ by $\{c_1, \cdots, c_n\}$. Color the regions of $L$ in checkerboard fashion, such that $R$ is colored white. For each crossing $c_i$, we assign two integers $a(c_i)$ and $w(c_i)$, according to the figure below.
\begin{center}
\includegraphics{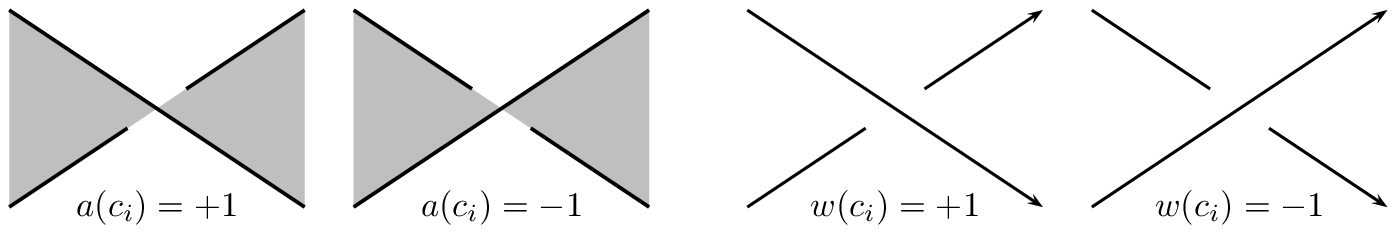} \centerline{\small Figure
4\quad}
\end{center}

The main theorem of this section can be described as below:
\begin{theorem}
Let $L$ be a diagram of a proper link, $L'$ is obtained by taking region crossing change on region $R$ of $L$, then
\begin{center}
\emph{Arf}$(L)+$\emph{Arf}$(L')=$
$\begin{cases}
0$ $ ($\emph{mod 2}$)& \text{\emph{if} $\frac{1}{2}\sum\limits_{i=1}^n(a(c_i)-w(c_i))=0$ $($\emph{mod}$ $ $4 )$;}\\
1$ $ ($\emph{mod 2}$)& \text{\emph{if} $\frac{1}{2}\sum\limits_{i=1}^n(a(c_i)-w(c_i))=2$ $($\emph{mod}$ $ $4 )$.}
\end{cases}$
\end{center}
Here $\{c_1, \cdots, c_n\}$ denote the crossing points on the boundary of $R$.
\end{theorem}

If we denote $\frac{1}{2}\sum\limits_{i=1}^n(a(c_i)-w(c_i))$ by $A(R)$, in fact $A(R)$ is an even integer $($see the proof below$)$, and now the equality above can be written as
\begin{center}
Arf$(L)+$Arf$(L')=$
$\begin{cases}
0$ $ ($mod 2$)& \text{if $A(R)=0$ $(${mod} 4$ )$;}\\
1$ $ ($mod 2$)& \text{if $A(R)=2$ $(${mod} 4$ )$.}
\end{cases}$
\end{center}
We remark that when $R$ is the changed region in $\sharp$-operation, it is easy to find that $A(R)=2$, therefore after one $\sharp$-operation the Arf invariant will change. Note that the equality above is also valid for $n$-gon move defined in [1].

Next we give the proof of Theorem 5.1.

\begin{proof}
The idea of the proof basically comes from the related result in [4] and [5]. The key point is that with the given region $R$, there exists an $n$-component proper link $L_R$ which is completed determined by $\{a(c_1), \cdots, a(c_n)\}$ and $\{w(c_1), \cdots, w(c_n)\}$, such that Arf$(L)+$Arf$(L')=$ Arf$(L_R)$ $($mod 2$)$. The figure below shows how to find this proper link $L_R$.
\begin{center}
\includegraphics{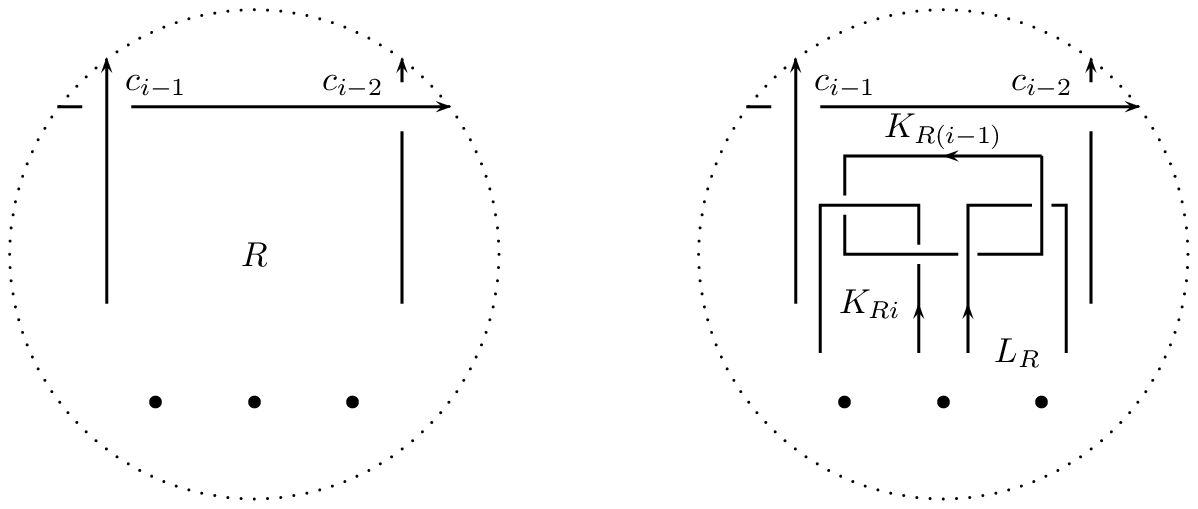} \centerline{\small Figure
5\quad}
\end{center}

Since $L_R$ is an $n$-component link, we can suppose $L_R=K_{R1}\cup K_{R2}\cup \cdots\cup K_{Rn}$ as above, then $lk(K_{R1}, K_{R2})=\pm1, \cdots, lk(K_{R(n-1)}, K_{Rn})=\pm1, lk(K_{Rn}, K_{R1})=\pm1$. Note that after taking $n$ connected sum operations between $L$ and $L_R$, we obtain $L'$. See the figure below. According to [7], it follows that Arf$(L)+$Arf$(L')=$ Arf$(L_R)$ $($mod 2$)$. Hence it suffices to find out the Arf invariant of $L_R$.
\begin{center}
\includegraphics{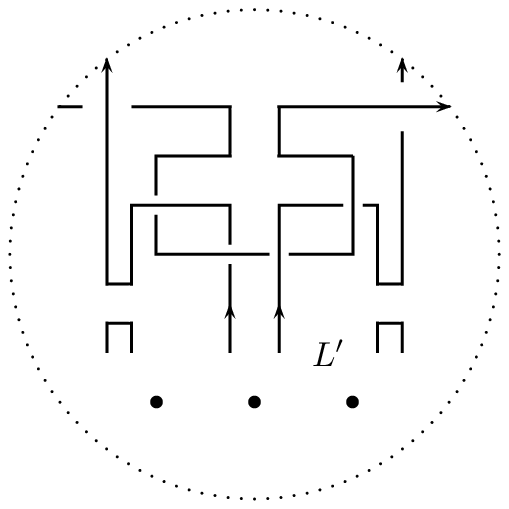} \centerline{\small Figure
6\quad}
\end{center}

As we mentioned before, it order to calculate Arf$(L_R)$, we just need to smooth $n-1$ crossing points from $K_{R1}\cap K_{R2}, \cdots, K_{R(n-1)}\cap K_{Rn}$ according to their orientations, then we will obtain a knot $K_R$ which has the same Arf invariant with $L_R$. Assign each $c_i$ with a pair of integers $(a(c_i), w(c_i))$, there are totally four cases for all $\{c_1, \cdots, c_n\}$, i.e. $(-1, +1), (+1, -1), (+1, +1), (-1, -1)$. Let $m_{-+}, m_{+-}, m_{++}$ and $m_{--}$ denote the number of the crossing points of these four types respectively. We claim that $K_R$ can be described as one of the four cases $($or their inverses$)$ below:
\begin{center}
\includegraphics{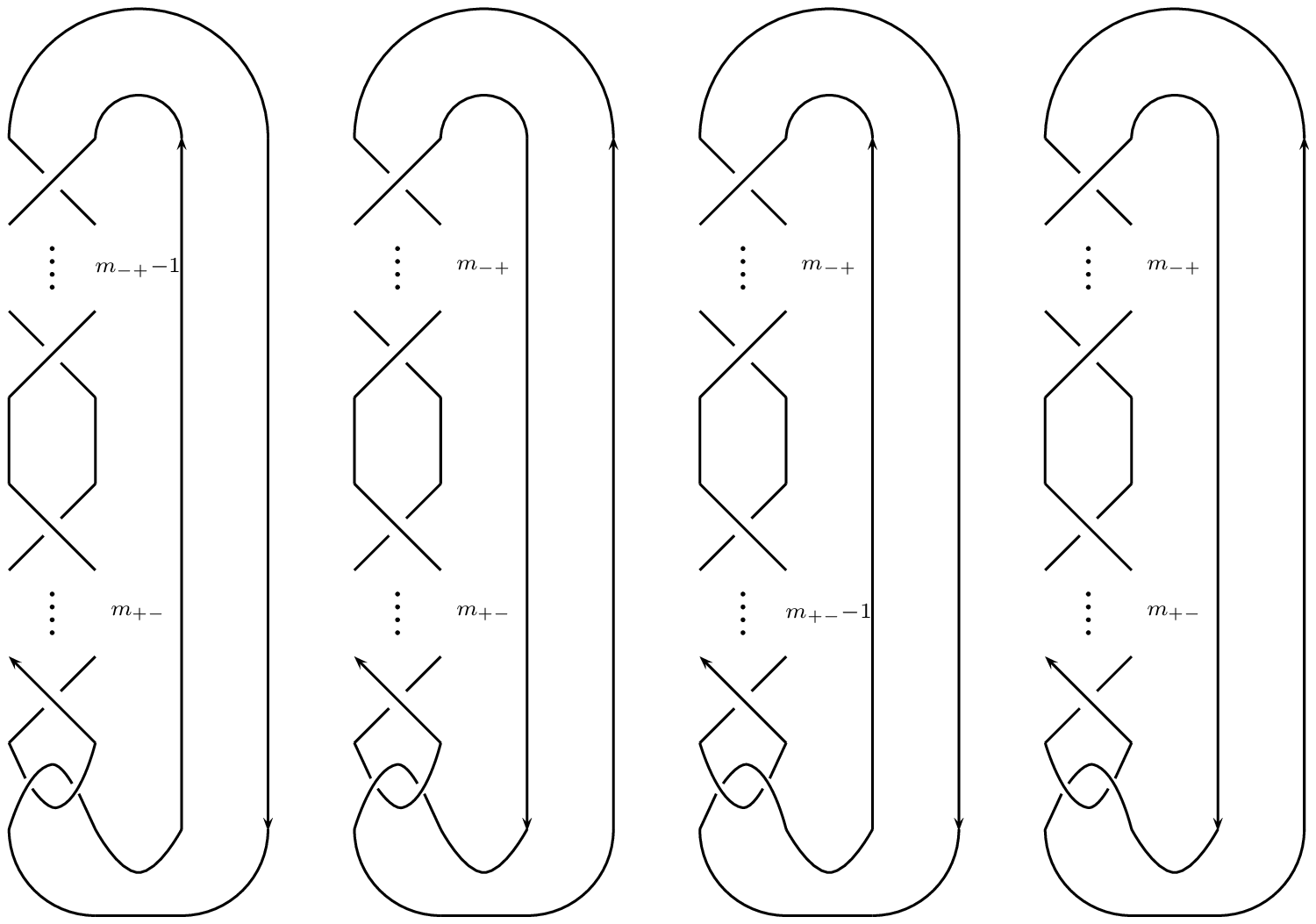} \centerline{\small Figure
7\quad}
\end{center}

In order to see this, it suffices to notice that for a crossing $c_i$ of type $(+1,+1)$ or $(-1,-1)$, smoothing one crossing point between $K_{Ri}$ and $K_{R(i+1)}$ will provide no twist. However if $c_i$ is of type $(-1,+1)$ or $(+1,-1)$, the same operation will increase one positive half-twist or one negative half-twist respectively. Since $K_R$ is a knot, from Figure 7 it is obvious that $m_{-+}+m_{+-}$ is an even integer, it follows that $A(R)=\frac{1}{2}\sum\limits_{i=1}^n(a(c_i)-w(c_i))=m_{+-}-m_{-+}$ is even. Because two full-twists preserve the Arf invariant, it follows that
\begin{center}
Arf$(K_R)=$
$\begin{cases}
0& \text{if $m_{-+}-m_{+-}=0$ $(${mod} 4$ )$;}\\
1& \text{if $m_{-+}-m_{+-}=2$ $(${mod} 4$ )$.}
\end{cases}$
\end{center}
The proof is complete.
\end{proof}

As a corollary, we have
\begin{corollary}
Let $L$ be a diagram of a proper link, $\{R_1, \cdots, R_n\}$ some regions of $L$, such that taking region crossing changes on $\{R_1, \cdots, R_n\}$ will turn $L$ to be trivial. Then
\begin{center}
\emph{Arf}$(L)=$
$\begin{cases}
0& \text{\emph{if} $\sum\limits_{i=1}^nA(R_i)=0$ $($\emph{{mod} 4}$ )$;}\\
1& \text{\emph{if} $\sum\limits_{i=1}^nA(R_i)=2$ $($\emph{{mod} 4}$ )$.}
\end{cases}$
\end{center}
\end{corollary}

\textbf{Acknowledgement} The authors wish to thank Professor Gao Hongzhu, Professor Lorenzo Traldi and Ayaka Shimizu for their useful suggestions and comments.

\no
\vskip0.2in
\no {\bf References}
\vskip0.1in

\footnotesize
\REF{[1]}Haruko Aida, {\it Unknotting operation for Polygonal type}. Tokyo J. Math. Vol. 15, No. 1, 111-121, 1992
\REF{[2]}Cheng Zhiyun, Gao Hongzhu, {\it On region crossing change and incidence matrix}. math.GT/1101.1129v2, 2011. To appear in Science China Mathematics.
\REF{[3]}Hoste, J., Nakanishi, Y. and Taniyama, K., {\it Unknotting operations involving trivial tangles}. Osaka J. Math. 27, 555-566, 1990
\REF{[4]}H. Murakami, {\it Some metrics on classical knots}. Math. Ann. 270, 35-45, 1985
\REF{[5]}H. Murakami, Y. Nakanishi, {\it On a certain move generating link-homology}. Math. Ann. 284, 75-89, 1989
\REF{[6]}Y. Nakanishi, {\it Replacements in the Conway third identity}. Tokyo J. Math. 14, 197-203, 1991
\REF{[7]}R. Robertello, {\it An invariant of knot cobordism}. Commun. Pure Appl. Math. 18, 543-555, 1965
\REF{[8]}D. Rolfsen, {\it Knots and links}. Publish or Perish, Inc. 1976
\REF{[9]}Ayaka Shimizu, {\it Region crossing change is an unknotting operation}. math.GT/1011.6304v2, 2010
\REF{[10]}Junming Xu, {\it Theory and Application of Graphs}. Kluwer Academic Publishers, 2003
\end{document}